\documentclass[11pt]{article}

\usepackage{tikz}
\usepackage{subfigure}
\usepackage[english]{babel}

\usepackage[center]{caption2}
\usepackage{amsfonts,amssymb,amsmath,latexsym,amsthm}
\usepackage{multirow}
\usepackage[usenames,dvipsnames]{pstricks}
\usepackage{epsfig}
\usepackage{pst-grad} 
\usepackage{pst-plot} 
\usepackage[space]{grffile} 
\usepackage{etoolbox} 
\makeatletter 
\patchcmd\Gread@eps{\@inputcheck#1 }{\@inputcheck"#1"\relax}{}{}

\def\qedB{{\hfill\enspace\vrule height8pt depth0pt width8pt}}

\newtheorem{claim}{Claim}

\begin{document}

\title{\bf\Large On a conjecture of Bondy and Vince}

\date{}

\author{
Jun Gao\footnote{Email: gj0211@mail.ustc.edu.cn.}~~~~~~~
Jie Ma\footnote{Email: jiema@ustc.edu.cn.
Research supported in part by NSFC grants 11501539 and 11622110 and Anhui Initiative in Quantum Information Technologies grant AHY150200.}\\
\medskip \\
School of Mathematical Sciences\\
University of Science and Technology of China\\
Hefei, Anhui 230026, China.
}

\maketitle

\begin{abstract}
Twenty years ago Bondy and Vince conjectured that for any nonnegative integer $k$, except finitely many counterexamples,
every graph with $k$ vertices of degree less than three contains two cycles whose lengths differ by one or two.
The case $k\leq 2$ was proved by Bondy and Vince, which resolved an earlier conjecture of Erd\H{o}s et. al..
In this paper we confirm this conjecture for all $k$.
\end{abstract}



The study on cycles has been extensive in the literature and
various problems on the existence of cycles with specified lengths were investigated decades ago (e.g., Erd\H{o}s \cite{Erd95, Erd97}).
One such question, proposed by Erd\H{o}s et. al. (see the discussion in \cite{BV}), asked whether every graph with minimum degree at least
three contains two cycles whose lengths differ by one or two.
This was answered affirmatively by Bondy and Vince \cite{BV} in the following stronger theorem:
with the exception of $K_1$ and $K_2$, every graph having at most two vertices of degree less than three contains two cycles of lengths differing by one or two.
They further conjectured the following generalization. (All graphs referred here are simple.)

\medskip

{\bf \noindent Conjecture.} (Bondy and Vince \cite{BV})
Let $k$ be any nonnegative integer. With finitely many exceptions, every graph having at most $k$ vertices of degree less than three has two cycles whose lengths differ by one or two.

\medskip

The authors of \cite{BV} remarked that they can also prove for $k=3$ (with twelve exceptions)
and the case $k=4, 5$ likely can be solved using their analysis.
We mention that the theorem in \cite{BV} was generalized in another direction by Fan \cite{Fan}.
For related topics, results and open problems, we refer interested readers to \cite{HS98,LM,SV08,V00,V16} and their references.

In this paper, we confirm the above conjecture of Bondy and Vince by the following.

\medskip

{\bf \noindent Theorem.}
Every graph, having at most $k$ vertices of degree less than three and at least $5k^2$ vertices, contains two cycles whose lengths differ by one or two.
		
\medskip

This also can be viewed as a resilience type answer to the above question of Erd\H{o}s et. al..
Let $G$ be an $n$-vertex graph with minimum degree at least three.
Then one can derive that by deleting any $\sqrt{n}/5$ edges from $G$, the remaining graph still contains two cycles of lengths differing by one or two.
Also by repeating the following procedure: first apply this theorem to find a pair of
two cycles of lengths differing by one or two and then delete two edges to destroy these two cycles,
one can in fact find $\Omega(\sqrt{n})$ such pairs of cycles in $G$.


For the proof, we will need a lemma of Bondy and Vince \cite{BV} (the proof of which uses an argument based on Thomassen and Toft \cite{Non-separating}).
Let $C$ be a cycle in a graph $G$.
A {\it bridge} of $C$ is either a chord of $C$ or a subgraph of $G$ obtained from a component $B$ in $G-V(C)$ by adding all edges between $B$ and $C$.
We call vertices of the bridge not in $C$ {\it internal}.

\medskip

{\bf \noindent Lemma.} (\cite{BV})
Let $G$ be a 2-connected graph, not a cycle, and let $C$ be an induced cycle in $G$ some bridge $B$ of which has as many internal vertices as possible.
Then either $B$ is the only bridge of $C$, or else that $B$ is a bridge containing exactly two vertices $u,v$ in $V(C)$ and every other bridge of $C$ is a path from $u$ to $v$.

\medskip

We are ready to present the proof of our result.

\medskip

{\noindent \bf Proof of Theorem.}
Throughout this proof, let $\mathcal{B}(G)$ denote the set of all vertices with degree at most two in a graph $G$,
and we say a pair of cycles is {\it good} if their lengths differ by one or two.
Let $f(1)=f(2)=3, f(3)=14, f(4)=56, f(5)=116$ and $f(k)=5k^2$ for $k\geq 6$
so that $\{f(k)-f(k-1)\}$ is strictly increasing and
\begin{align}\label{equ:f}
f(k)\geq f(k-1)+7k+f(3) \mbox{ holds for all }k\geq 4.
\end{align}
We will prove by induction on $k$ that every graph $G$ with $|\mathcal{B}(G)|\leq k$ and at least $f(k)$ vertices contains a good pair of cycles.
The case $k\leq 2$ follows by the aforementioned theorem of Bondy and Vince \cite{BV},
so we may assume that $k\geq 3$ and the statement holds for any integer $\ell<k$.

We begin by showing some useful facts on graphs $H$ with $|\mathcal{B}(H)|=k$.
For a subgraph $F$ of $H$, by $H-F$ we denote the graph obtained from $H$ by deleting all vertices of $F$.

\begin{claim}\label{Claim: 1}
Let $H$ be a graph with $|\mathcal{B}(H)|= k$, minimum degree $\delta(H)\geq 2$ and no good pair of cycles.
If $k=3$ or $k\geq 4$ and $|V(H)|\geq f(k-1)+f(3)$, then $H$ is 2-connected.
\end{claim}

\begin{proof}
First we show that $H$ is connected. Otherwise, there exist at least two components of $H$.
Let $D_1$ be one of the components of $H$ and let $D_2=H-D_1$.
Let $j=|\mathcal{B}(D_1)|$. So $|\mathcal{B}(D_2)|=k-j$.
Since $\delta(H)\geq 2$, we see $|V(D_i)|\geq 3$ for each $i\in [2]$.
By the theorem of Bondy and Vince, we may assume that $3\leq j\leq k-3$.
Since $\{f(k)-f(k-1)\}$ is strictly increasing,
$|V(D_1)|+|V(D_2)|=|V(H)|\geq f(k-1)+f(3)\geq f(j)+f(k-j)$.
So either $|V(D_1)|\geq f(j)$ or $|V(D_2)|\geq f(k-j)$.
By induction, in either case we can find a good pair of cycles. This shows that $H$ is connected.

Suppose that there exists a cut-vertex $u$ in $H$.
Let $B_1$ be a component of $H-\{u\}$ and $B_2=H-\{u\}-B_1$.
For $i\in [2]$, let $b_i=|\mathcal{B}(H)\cap B_i|$.
It is clear that $k-1\leq b_1+b_2\leq k$.
We also have $b_i\geq 2$. Indeed, otherwise say some $b_i\leq 1$,
then $H_i=H[B_i\cup \{u\}]$ is a graph with at least three vertices (since $\delta(H)\geq 2$) and at most two vertices of degree less than three;
by the theorem of Bondy and Vince, it contains a good pair of cycles (so does $H$), a contradiction.
This already proves for $k=3$ (as it is impossible to have $b_1+b_2\geq 4$).
Now we consider $k\geq 4$.
Since $H_i$ contains no good pair of cycles and $|\mathcal{B}(H_i)|\leq b_i+1\leq k-1$,
by induction we have
$f(b_1+1)+f(b_2+1)> |V(H_1)|+|V(H_2)|>|V(H)|\geq f(k-1)+f(3)=\max_{2\leq s\leq k/2} \{f(s+1)+f(k-s+1)\}$\footnote{Here the equality follows from the fact that $f(k)-f(k-1)$ is increasing.},
again a contradiction. This finishes the proof.			
\end{proof}

A cycle $C$ in a graph $H$ is called {\it feasible} if it is induced and $H-C$ is connected.

\begin{claim}\label{Claim: 2}
Let $H$ be a 2-connected graph with $|\mathcal{B}(H)|=k$ and no good pair of cycles.
If $|V(H)|\geq f(k-1)+1$, then there exists a feasible cycle in $H$.
\end{claim}

\begin{proof}
We see that $H$ is not a cycle.
Let $C$ be an induced cycle in $H$ such that some bridge $B$ of $C$ has the maximum number of internal vertices.
If $B$ is the only bridge of $C$, then $C$ is feasible in $H$ and we are home.
Therefore, by Lemma we may assume that $B$ is a bridge containing exactly two vertices $u,v \in V(C)$ and
there exists another bridge $P$ of $C$ which is a path from $u$ to $v$.
Let $P = v_0...v_t$ with $u=v_0$ and $v=v_t$, where we have $t\geq 2$, $d_H(u)\geq 4$ and $d_H(v)\geq 4$.
Let $R= H-\{v_1\}$. Then $|V(R)|\geq f(k-1)$ and $|\mathcal{B}(R)|\leq k-1$.
By induction, $R$ (and thereby $H$) has a good pair of cycles, a contradiction.
\end{proof}

For $F, F'\subseteq H$, let $N_F(F')$ be the set of vertices in $F$ adjacent to some vertex in $F'$.

\begin{claim}\label{Claim: 3}
Let $H$ be a 2-connected graph with $|\mathcal{B}(H)|=k$ and no good pair of cycles,
whose order is at least $f(2)+4$ for $k=3$ and at least $f(k-1)+f(3)+2k$ for $k\geq 4$.
Let $C$ be any feasible cycle in $H$ and let $A=N_C(H-C)$.
Then $C$ has length at most $2k$ and divisible by four, whose vertices alternate between $A$ and $\mathcal{B}(H)\cap V(C)$.
\end{claim}

\begin{proof}
Let $C =u_1u_2...u_ru_1$. First we show that
there is no pair $u_i,u_{i+\lfloor r/2 \rfloor+1}\in A$ where the subscript is modulo $r$.
Otherwise, there exists a path $P$ with endpoints $u_i, u_{i+\lfloor r/2 \rfloor+1}$ and all internal vertices in $H-C$,
which together with the two segments of $C$ between $u_i$ and $u_{i+\lfloor r/2 \rfloor+1}$ form a good pair of cycles in $H$.
This also implies that $|A|\leq |C|/2$.
Now suppose that $|A|<|C|/2$.
Since $C$ is induced, we see that $V(C)$ is partitioned into $A$ and $\mathcal{B}(H)\cap V(C)$.
Let $H'$ be obtained from $H$ by deleting $\mathcal{B}(H)\cap V(C)$.
Then $|\mathcal{B}(H')|\leq k-|\mathcal{B}(H)\cap V(C)|+ |A|\le k-1$ and,
as $|\mathcal{B}(H)\cap V(C)|\leq k$, we have $|V(H')|\geq |V(H)|-k\geq f(k-1)$.
By induction, $H'$ contains a good pair of cycles, a contradiction.
So $|A|=|\mathcal{B}(H)\cap V(C)|=|C|/2$ and $C$ is an even cycle of length at most $2k$.
Let $r=2s$.

Suppose there exist two consecutive vertices in $C$, say $u_1,u_2$, belonging to $A$.
As noted above there is no $u_i,u_{i+s+1} \in A$ for any $i$.
So $u_s, u_{s+1}, u_{s+2}, u_{s+3}$ are not in $A$ and
the pair $u_i, u_{i+s+1}$ for any $3\leq i\leq s-1$ has at most one vertex in $A$,
which together imply that $|A|\leq s-1$, a contradiction.
This shows that no two consecutive vertices of $C$ can be in $A$ and
thus the vertices of $C$ alternate between $A$ and $\mathcal{B}(H)\cap V(C)$.
Finally let us note that $r$ is divisible by four,
as otherwise $s$ is odd, so if $u_i\in A$ then $u_{i+s+1}\in A$, a contradiction.
\end{proof}

\begin{claim}\label{Claim:4}
Let $H, C, A$ be from Claim \ref{Claim: 3}.
Let $D=\{v\in V(C)|~d_H(v)\leq 3\}$ and $A'=\mathcal{B}(H-D)\setminus \mathcal{B}(H)$.
Then $|V(C)\backslash D|\leq 1$, $|A|=|A'|=|C|/2$ with $A\cap A'=V(C)\backslash D$,
and every vertex in $A-V(C)\backslash D$ is adjacent to a unique vertex in $A'-V(C)\backslash D$;
moreover, $H-D$ is 2-connected with $|\mathcal{B}(H-D)|=|\mathcal{B}(H)|$.
\end{claim}

\begin{proof}	
Suppose that there are two vertices say $u_j,u_{j+t}\in V(C)\backslash D$.
So $u_j, u_{j+t}\in A$ have degree at least four in $H$, $t\in [2, k]$ is even and we may assume that $d_H(u_{j+i})\leq 3$ for all $1\leq i\leq t-1$.
Let $H'$ be obtained from $H$ by deleting $S=\{u_{j+1},...,u_{j+t-1}\}$.
By Claim \ref{Claim: 3} we have $\mathcal{B}(H')\leq k-|S\cap \mathcal{B}(H)|+|S\cap A|=k-1$ and $|V(H')|\geq |V(H)|-k\geq f(k-1)$,
so $H'$ has a good pair of cycles. This contradiction proves that $|V(C)\backslash D|\leq 1$.
Let $|C|=2s$.

Consider $H-D$. We have that $|V(H-D)|$ is at least $f(2)$ for $k=3$ (by Claim \ref{Claim: 3}, $|D|\leq |C|=4$) and is at least $f(k-1)+f(3)$ for $k\geq 4$,
and $\mathcal{B}(H-D)=(\mathcal{B}(H)-\mathcal{B}(H)\cap V(C))\cup A'$.
So $|\mathcal{B}(H-D)|=k-s+|A'|$.
First we see $|A'|\geq s$; as otherwise $|\mathcal{B}(H-D)|\leq k-1$
and by induction $H-D$ has a good pair of cycles.
If $V(C)\backslash D=\emptyset$, then there are exactly $|A|=s$ edges between $C$ and $H-C$,
so the deletion of $D=V(C)$ will result in at most $s$ new vertices of degree at most two, that is $|A'|\leq s$.
Therefore in this case, $|A'|=s=|A|$ and $A\cup A'$ induces a matching of size $s$ in $H$.
Now suppose $V(C)\backslash D=\{u\}$.
There are exactly $s-1$ edges between $D$ and $H-C$,
so the deletion of $D$ will result in at most $s-1$ new vertices of degree at most two in $H-C$,
which, plus possibly $u\in A'$ (if and only if $d_H(u)=4$), show that $|A'|\leq s$.
Therefore, again we have $|A'|=s=|A|$; in fact $d_H(u)=4$, $A\cap A'=\{u\}$ and $A\cup A'-\{u\}$ induces a matching of size $s-1$ in $H$.

The above analysis also demonstrates that every vertex in $A'$ has degree two in $H-D$ and thus $\delta(H-D)\geq 2$.
By Claim \ref{Claim: 1}, $H-D$ is 2-connected. This finishes the proof.
\end{proof}

Let $G_0$ be a graph with $|\mathcal{B}(G_0)|=k$ and at least $f(k)$ vertices. Suppose for a contradiction that $G_0$ has no good pair of cycles.

We will define a sequence of subgraphs $G_0\supseteq G_1\supseteq... \supseteq G_m$.
First we show there exists a 2-connected $G_1\subseteq G_0$ with $|\mathcal{B}(G_1)|=k$ and $|V(G_1)|\ge f(k)-k$.
Let $S$ be the set of vertices of degree at most one in $G_0$ and let $G_1=G_0-S$.
Then $|\mathcal{B}(G_1)|\leq |\mathcal{B}(G_0)|-|S|+|S'|$ where $S'$ denotes the set of vertices in $G_1$ adjacent to $S$ (so clearly $|S'|\leq |S|$).
We must have $|\mathcal{B}(G_1)|=k$; otherwise, $|\mathcal{B}(G_1)|\leq k-1$ and $|V(G_1)|\geq f(k)-k\geq f(k-1)$, implying a good pair of cycles in $G_1\subseteq G_0$.
This shows that $S\cup S'$ induces a matching of size $|S|$ in $G_0$ and every vertex in $S'$ is in $\mathcal{B}(G_1)\backslash \mathcal{B}(G_0)$.
Therefore $\delta(G_1)\geq 2$. As $|V(G_1)|\geq f(k)-k\geq f(k-1)+f(3)$ for $k\geq 4$, by Claim \ref{Claim: 1} we conclude that $G_1$ is 2-connected.

Now suppose we have defined $G_i$ for some $i\geq 1$.
If $G_i$ is 2-connected with $|\mathcal{B}(G_i)|=k$ and of order at least $f(2)+4$ for $k=3$ and at least $f(k-1)+f(3)+2k$ for $k\geq 4$,
then
\begin{itemize}
\item let $C_i$ be a feasible cycle in $G_i$ (by Claim \ref{Claim: 2}), with the preference to be a four-cycle,
\item let $A_i=N_{C_i}(G_i-C_i)$, and further
\item let $D_i=\{v\in V(C_i)|~d_{G_i}(v)\leq 3\}$, $G_{i+1}=G_i-D_i$ and $A'_i=\mathcal{B}(G_{i+1})\setminus \mathcal{B}(G_i)$.
\end{itemize}
Otherwise we terminate.
Let $G_0\supseteq G_1\supseteq... \supseteq G_m$ be the sequence of subgraphs we obtain as above.
Then we can apply Claims \ref{Claim: 3} and \ref{Claim:4} for $G_i, C_i, A_i$ for each $1\leq i\leq m-1$.
In particular, we see that $G_i$ for all $1\leq i\leq m$ is 2-connected with $|\mathcal{B}(G_i)|=k$.
So the reason we terminate at $G_m$ is because the order of $G_m$ is at most $f(2)+3$ for $k=3$ and at most $f(k-1)+f(3)+2k-1$ for $k\geq 4$.
By \eqref{equ:f}, this implies that $\sum_{i=1}^{m-1} |D_i|=|V(G_1)|-|V(G_m)|$ is at least $5$ for $k=3$
and at least $(f(k)-k)-(f(k-1)+f(3)+2k-1)> 4k$ for $k\geq 4$.

In what follows, we will investigate properties on the cycles $C_1,...,C_{m-1}$,
which eventually will lead a contradiction to the above lower bound of $\sum_{i=1}^{m-1} |D_i|$.

We first show that if $A_i'\not\subseteq \mathcal{B}(G_{i+1})\cap V(C_{i+1})$,
then $C_{i+1}$ is a feasible cycle in $G_i$.
Clearly we have $A_i'\cap A_{i+1}=\emptyset$ (as any vertex in $A_i'$ has degree two in $G_{i+1}$,
while vertices in $A_{i+1}$ have degree at least three in $G_{i+1}$).
Note that $V(C_{i+1})$ is partitioned into $\mathcal{B}(G_{i+1})\cap V(C_{i+1})$ and $A_{i+1}$.
So $A_i'\not\subseteq \mathcal{B}(G_{i+1})\cap V(C_{i+1})$ implies that $A_i'$ intersects with $G_{i+1}-C_{i+1}$.
By Claim \ref{Claim:4}, we see that any vertex in $A_i'$ is adjacent to $D_i$,
so $G_{i+1}-C_{i+1}$ is adjacent to $D_i$.
Therefore, $G_i-C_{i+1}$ is connected and thus $C_{i+1}$ is feasible in $G_i$.

We now assert that there exists some $t$ such that $|C_1|=...=|C_t|=4$ and $|C_i|\geq 8$ for each $t+1\leq i\leq m-1$.
By Claim \ref{Claim: 3} each $|C_i|$ is divisible by four,
so it suffices to prove that if $|C_i|\geq 8$, then $|C_{i+1}|\geq 8$.
Suppose this is not true. Then $|C_{i+1}|=4<|C_i|$,
implying that $|A_i'|=|C_i|/2>|C_{i+1}|/2=|\mathcal{B}(G_{i+1})\cap V(C_{i+1})|$.
So $C_{i+1}$ is a feasible cycle (of length four) in $G_i$.
But this contradicts our preference for choosing $C_i$ in $G_i$, finishing the proof.

Next we claim that for each $1\leq i\leq t-1$, $C_{i+1}$ is feasible in $G_i$.
We have $|C_i|=|C_{i+1}|=4$. So $|A_i'|=2=|\mathcal{B}(G_{i+1})\cap V(C_{i+1})|$.
If $A_i'=\mathcal{B}(G_{i+1})\cap V(C_{i+1})$,
then this will give two possible configurations between $C_i$ and $C_{i+1}$
and in each case we can find a good pair of cycles easily
(of lengths 4 and 5, or 4 and 6).
So we may assume $A_i'\neq \mathcal{B}(G_{i+1})\cap V(C_{i+1})$.
Therefore, $C_{i+1}$ is a feasible cycle in $G_i$.

We also claim that for each $i\geq t+1$, $C_{i+1}$ is feasible in $G_i$.
Let $C_i=u_0u_1...u_{s-1}u_0$ where $s\geq 8$. So $|A_i|=|A_i'|\geq 4$.
By Claims \ref{Claim: 3} and \ref{Claim:4},
we may assume that $u_0,u_2,u_4\in A_i-V(C_i)\backslash D_i$
and $x,y,z\in A_i'-V(C_i)\backslash D_i$ with $u_0x, u_2y, u_4z\in E(G_i)$.
We may also assume that $x,y,z\in A_i'\subseteq \mathcal{B}(G_{i+1})\cap V(C_{i+1})$
(as otherwise, $C_{i+1}$ is feasible in $G_i$).
Let $P$ denote one of the two segments of $C_{i+1}$ between $x$ and $y$ with $|E(P)|\geq |C_{i+1}|/2$.
Clearly $|E(P)|\geq 4$ is even.
If $|E(P)|\equiv 2$ mod $4$, then $C'=P\cup xu_0u_1u_2y$ forms a cycle in $G_i$ whose length is $2$ mod $4$ and thus is at least 10.
Then there exist two vertices $a,b\in V(P)\cap A_{i+1}$ which divide the cycle $C'$ into two segments of lengths differing by two.
Also there exists a path with endpoints $a,b$ and all internal vertices in $G_{i+1}-C_{i+1}$, which is internally disjoint from $C'$.
Putting the above together, we find a good pair of cycles in $G_i$.
So $|E(P)|\equiv$ 0 mod $4$.
Similarly any segment of $C_{i+1}$ between $y$ and $z$ has length $0$ mod $4$.
Therefore any segment (say $Q$) of $C_{i+1}$ between $x$ and $z$ has length $0$ mod $4$ as well.
Then $Q\cup xu_0u_1u_2u_3u_4z$ forms a cycle of length $2$ mod $4$.
If $|E(Q)|=4$, then $C_i$ and $(C_i-\{u_1,u_2,u_3\})\cup u_0x\cup Q\cup zu_4$ form a good pair of cycles (whose lengths differ by two).
So $|E(Q)|\geq 8$. Again there exist two vertices $a,b\in V(Q)\cap A_{i+1}$ dividing the cycle $Q\cup xu_0u_1u_2u_3u_4z$
into two segments of lengths differing by two.
By similar arguments as above this enables us to find a good pair of cycles in $G_i$, proving the claim.

Finally we are completing the proof using the above two claims.
We point out that for each $1\leq i\leq t-1$, $C_{i+1}$ and $C_i$ play the same role in $G_i$.
Repeatedly applying this, one can conclude that in fact all four-cycles $C_1, C_2, ...,C_t$ are feasible in $G_1$.
By Claims \ref{Claim: 3} and \ref{Claim:4},
each $V(C_i)\backslash A_i=\mathcal{B}(G_1)\cap V(C_i)$ provides $|C_i|/2\geq |D_i|/2$ distinct vertices in $\mathcal{B}(G_1)$.
So
\begin{align*}
\sum_{i=1}^t |D_i|\leq \sum_{i=1}^t 2|\mathcal{B}(G_1)\cap V(C_i)|\leq 2|\mathcal{B}(G_1)|\leq 2k.
\end{align*}
Similarly, $C_{t+1}, C_{t+2},...,C_{m-1}$ are feasible in $G_{t+1}$ and using this fact, we can derive that
\begin{align*}
\sum_{i=t+1}^{m-1} |D_i|\leq \sum_{i=t+1}^{m-1} 2|\mathcal{B}(G_{t+1})\cap V(C_i)|\leq 2|\mathcal{B}(G_{t+1})|\leq 2k.
\end{align*}
Putting the above together, we have $\sum_{i=1}^{m-1} |D_i|\leq 4k$, but we have seen that $\sum_{i=1}^{m-1} |D_i|> 4k$ for $k\geq 4$.
This contradiction finishes the proof for the case $k\geq 4$.
For $k=3$, by Claim \ref{Claim: 3}, all feasible cycles must be of length four, so $t=m-1$.
And we obtained that $\sum_{i=1}^t |D_i|\geq 5$.
This tells that there are at least two four-cycles $C_1, C_2$, each providing two vertices in $\mathcal{B}(G_1)$.
But it is impossible as $3=|\mathcal{B}(G_1)|\geq 4$.
The proof now is completed. \qedB

\medskip

We didn't make lots of efforts to optimise the constant in the proof as we tend to believe that
the quadratic bound $O(k^2)$ can be further improved (perhaps, to a linear term $O(k)$).

To conclude this paper we would like to mention a conjecture of \cite{LM},
which seems to be a natural generalization for the question of Erd\H{o}s et. al.
and has implications (if true) for other related problems:
For any nonnegative integer $k$, every graph with minimum degree at least $k+1$ contains $k$ cycles
$C_1,...,C_k$ with $|C_{i+1}|=|C_i|+d$ for $d\in \{1,2\}$.

\bigskip

\noindent {\bf Acknowledgement.}
We would like to thank the referees for their careful reading and helpful suggestions.

\bigskip

\noindent {\bf Note added in proof.}
After this paper was submitted for publication, a proof to the problem mentioned in the last paragraph was announced in \cite{GHLM}.


\begin{thebibliography}{99}
\bibitem{BV}
J. A. Bondy and A. Vince, \newblock{Cycles in a graph whose lengths differ by one or two}, \newblock{\emph{J. Graph Theory}} \textbf{27} (1998), 11--15.

\bibitem{Erd95}
P. Erd\H{o}s, \newblock{Some of my recent problems in combinatorial number theory, geometry and combinatorics},
in \emph{Graph Theory, Combinatorics and Algorithms}, Vol. 1, Proc. Seventh
Quadrennial International Conference on the Theory and Applications of Graphs
(Y. Alavi and A. Schwenk, Eds.), pp. 335--349, Wiley, New York, 1995.

\bibitem{Erd97}
P. Erd\H{o}s, \newblock{Some old and new problems in various branches of combinatorics},
in \emph{Graphs and Combinatorics} (Marseille, 1995). \newblock{\emph{Discrete Math.}} \textbf{165/166} (1997), 227--231.

\bibitem{Fan}
G. Fan, \newblock{Distribution of cycle lengths in graphs}, \newblock{\emph{J. Combin. Theory Ser. B}} \textbf{84} (2002), 187--202.

\bibitem{GHLM}
J. Gao, Q. Huo, C. Liu and J. Ma,
\newblock{A unified proof of conjectures on cycle lengths in graphs}, arXiv:1904.08126.


\bibitem{HS98}
R. H\"aggkvist and A. Scott, \newblock{Arithmetic progressions of cycles},
\newblock{\emph{Technical Report}} No. 16 (1998), Matematiska Institutionen, Ume\r{a}Universitet.

\bibitem{LM}
C. Liu and J. Ma, \newblock{Cycle lengths and minimum degree of graphs}, \newblock{\emph{J. Combin. Theory Ser. B}} \textbf{128} (2018), 66--95.


\bibitem{SV08}
B. Sudakov and J. Verstra\"ete, \newblock{Cycle lengths in sparse graphs}, \newblock{\emph{Combinatorica}} \textbf{28} (2008), 357--372.


\bibitem{Non-separating}
C. Thomassen and B. Toft, \newblock{Non-separating induced cycles in graphs},
\newblock{\emph{J. Combin. Theory Ser. B}} \textbf{31} (1981), 199--224.


\bibitem{V00}
J. Verstra\"ete, \newblock{On arithmetic progressions of cycle lengths in graphs}, \newblock{\emph{Combin.
Probab. Comput.}} \textbf{9} (2000), 369--373.

\bibitem{V16}
J. Verstra\"ete, \newblock{Extremal problems for cycles in graphs},
In \emph{Recent Trends in Combinatorics},
A. Beveridge et al. (eds.), The IMA Volumes in Mathematics and its Applications 159, 83--116, Springer, New York, 2016.


\end{thebibliography}
\end{document}